\documentclass[11pt,bezier]{article}
\usepackage{amsmath, amssymb, amsfonts, euscript, pifont, mathrsfs, latexsym, graphicx}

\newtheorem{prethm}{{\bf Theorem}}

\newenvironment{thm}{\begin{prethm}\sl{\hspace{-0.5
               em}{\bf.}}}{\end{prethm}}

\newtheorem{prepro}[prethm]{{\bf Proposition}}

\newtheorem{prelem}[prethm]{{\bf Lemma}}

\newenvironment{lem}{\begin{prelem}\sl{\hspace{-0.5
               em}{\bf.}}}{\end{prelem}}

\newtheorem{predeff}[prethm]{{\bf Definition}}

\newenvironment{deff}{\begin{predeff}\rm{\hspace{-0.5
               em}{\bf.}}}{\end{predeff}}

\newtheorem{precor}[prethm]{{\bf Corollary}}

\newtheorem{preconj}[prethm]{{\bf Conjecture}}

\newtheorem{preremark}[prethm]{{\bf Remark}}

\newenvironment{remark}{\begin{preremark}\rm{\hspace{-0.5
               em}{\bf.}}}{\end{preremark}}

\newtheorem{preexample}[prethm]{{\bf Example}}

\newtheorem{prethmm}{{\bf Theorem}}

\newtheorem{preproof}{{\bf\textsf{Proof.}}}

\newenvironment{proof}[1]{\begin{preproof}{\rm
               #1}\hfill{$\Box$}}{\end{preproof}}

\newcommand{\bmi}[1]{\mbox{\boldmath $ #1$}}

\makeatletter
\def\bbordermatrix#1{\begingroup \m@th
  \@tempdima 4.75\p@
  \setbox\z@\vbox{%
    \def\cr{\crcr\noalign{\kern2\p@\global\let\cr\endline}}%
    \ialign{$##$\hfil\kern2\p@\kern\@tempdima&\thinspace\hfil$##$\hfil
      &&\quad\hfil$##$\hfil\crcr
      \omit\strut\hfil\crcr\noalign{\kern-\baselineskip}%
      #1\crcr\omit\strut\cr}}%
  \setbox\tw@\vbox{\unvcopy\z@\global\setbox\@ne\lastbox}%
  \setbox\tw@\hbox{\unhbox\@ne\unskip\global\setbox\@ne\lastbox}%
  \setbox\tw@\hbox{$\kern\wd\@ne\kern-\@tempdima\left[\kern-\wd\@ne
    \global\setbox\@ne\vbox{\box\@ne\kern2\p@}%
    \vcenter{\kern-\ht\@ne\unvbox\z@\kern-\baselineskip}\,\right]$}%
  \null\;\vbox{\kern\ht\@ne\box\tw@}\endgroup}
\makeatother

\title{\bf\LARGE  Maximum order   of triangle-free graphs with a given rank\\ \vspace{1cm}}

\author{\large E. Ghorbani$^{\,\rm 1, 2}$ \quad \quad  A. Mohammadian$^{\,\rm 2}$ \quad \quad B. Tayfeh-Rezaie$^{\,\rm 2}$\\[.4cm]
{\sl $^{\rm 1}$Department of Mathematics, K.N. Toosi University of Technology,}\\
{\sl P.O. Box 16315-1618, Tehran, Iran}\\[0.3cm]
{\sl $^{\rm 2}$School of Mathematics, Institute for Research in Fundamental
Sciences (IPM),}\\{\sl P.O. Box
19395-5746, Tehran, Iran }
\\[0.5cm]{
$\mathsf{e\_ghorbani@ipm.ir}$ \quad\quad  $\mathsf{ali\_m@ipm.ir}$ \quad\quad  $\mathsf{tayfeh}$-$\mathsf{r@ipm.ir}$}}

\date{}

\begin{document}
\maketitle

\vspace{5mm}

\begin{abstract}
The rank of a graph is defined to be the rank of its adjacency matrix.
A graph is called  reduced if it has no isolated vertices and no two vertices  with the same set of neighbors.
We determine the maximum order of  reduced  triangle-free graphs with a
given rank and characterize all such  graphs achieving the maximum order.

\vspace{5mm}
\noindent {\bf Keywords:}   rank, triangle-free  graph, adjacency matrix \\[.1cm]
\noindent {\bf AMS Mathematics Subject Classification\,(2010):}   05C50, 05C75, 15A03
\end{abstract}

\vspace{5mm}

\section{Introduction}

For a  graph $G$, we denote  by  $V(G)$    the vertex set of $G$. The   {\sl order} of $G$ is defined as  $|V(G)|$. Let    $V(G)=\{v_1, \ldots , v_n\}$. The {\sl adjacency matrix} of $G$ is an $n \times  n$
 matrix $A(G)$ whose $(i, j)$-entry is $1$ if $v_i$ is adjacent to $v_j$ and  $0$ otherwise.
The {\sl rank} of  $G$, denoted by $\mathrm{rank}(G)$,  is the  rank of $A(G)$.

For a vertex $v$ of $G$, let $N(v)$ denote  the set of
all vertices of $G$ adjacent to $v$.  The {\sl degree} of  $v$ is defined as $|N(v)|$.
We say that  $G$ is {\sl reduced} if it has no isolated vertex  and no two vertices $u,v$ with $N(u)=N(v)$. Indeed, adding an isolated vertex or introducing a new vertex with the same neighbor set as an  existing  vertex  does not change the rank.
Let  $r\geqslant2$ be an integer. It is straightforward   to see  that every  reduced  graph of   rank $r$ has at most $2^r-1$ vertices  \cite{akb}. Let $m(r)$ be  the maximum possible order of a reduced  graph of  rank $r$.    Kotlov and   Lov\'asz   \cite{lov} proved that  there exists  a  constant $c$  such  that $m(r)\leqslant c\cdot2^{r/2}$ and for any  $r\geqslant2$ they  constructed a graph of rank $r$ and order
\begin{eqnarray*}
\mu(r)=\left\{\begin{array}{ll}
2^{(r+2)/2}-2 & \text{if $r$ is even}, \\
5\cdot2^{(r-3)/2}-2 & \text{if $r>1$ is odd}.
\end{array}\right.
\end{eqnarray*}
Akbari, Cameron  and  Khosrovshahi   \cite{akb}  conjectured that in fact $m(r)=\mu(r)$.
Haemers and  Peeters  \cite{ham}  proved the conjecture for  graphs  containing  an induced matching of size $r/2$  or an induced subgraph consisting  of a  matching of size $(r-3)/2$ and a cycle of length $3$.
Royle   \cite{roy} proved that the rank of every reduced graph containing no path of length $3$ as an
induced subgraph is equal to the order.

We   proved in \cite{gmt}    that every reduced tree  of rank $r$  has at most $t(r)=3r/2-1$ vertices and  characterized  all reduced  trees  of rank $r$ and order  $t(r)$. It  was also  shown    that every reduced bipartite graph  of rank $r$  has at most $b(r)=2^{r/2}+r/2-1$ vertices and  all reduced   bipartite graphs  achieving this bound were  determined.   Note that the rank of a bipartite graph is always  even. In this article, we prove  that  every reduced non-bipartite triangle-free graph  of rank $r$  has at most $c(r)=3\cdot2^{\lfloor r/2\rfloor-2}+\lfloor r/2\rfloor$ vertices and  characterize all reduced non-bipartite triangle-free graphs  of rank $r$ and order $c(r)$.

\section{Preliminaries}

For  a  graph $G$, a subset $S$ of  $V(G)$ with $|S|>1$  is called a {\sl duplication  class} of $G$  if  $N(u)=N(v)$,  for every   $u, v\in  S$.
For  a subset  $X$  of   $V(G)$,  the notation $G-X$ represents
the subgraph obtained by removing the vertices in $X$  from $G$.

\begin{lem}\label{jad} {\rm\cite{kot, lov}}
For  any  reduced graph $G$,    the following hold.
\begin{itemize}
\item [{\rm (i)}] For every vertex $v\in V(G)$, $\mathrm{rank}(G-N(v))\leqslant\mathrm{rank}(G)-2$.
\item [{\rm (ii)}]  For every non-adjacent vertices $u,v\in V(G)$, $\mathrm{rank}(G-(N(u)\triangle N(v)))\leqslant\mathrm{rank}(G)-2$, where
$\triangle$ denotes the  symmetric difference.
\end{itemize}
\end{lem}

The following lemma has a key role  in our proofs.

\begin{lem}\label{lov}
Let $G$ be a reduced graph  and $H$ be an induced  subgraph of $G$ with the  maximum  possible  order subject to    $\mathrm{rank}(H)<\mathrm{rank}(G)$. Then  $\mathrm{rank}(H)\geqslant\mathrm{rank}(G)-2$ and the equality occurs if $H$ is not reduced. Moreover,  the following  properties hold.
\begin{itemize}
\item [{\rm (i)}] $|V(G)\setminus V(H)|\leqslant\min\{|N(u)\triangle N(v)|\, |\,  u, v\in V(G)\}\cup\{|N(u)|\, |\, u\in V(G)\}$.
\item [{\rm (ii)}]   If $w$   is  an isolated vertex of $H$, then $N(w)=V(G)\setminus V(H)$.
\item [{\rm (iii)}] Each   duplication  class of  $H$  has    two elements and $H$ has at most  one isolated vertex.
\item [{\rm (iv)}]   One may label the duplication classes of $H$, if any,  as    $\{v_1, v_1'\}, \ldots, \{v_s, v_s'\}$ so that  there exist two  disjoint sets $T_1$ and  $T_2$  such that $V(G-H)=T_1\cup T_2$, $T_1\subseteq  N(v_i)\setminus  N(v_i')$ and  $T_2\subseteq  N(v_i')\setminus  N(v_i)$, for all $i\in\{1, \ldots, s\}$.
\end{itemize}
Furthermore, if   $H$  is    an induced  subgraph of $G$ with the  maximum  possible  order subject to    $\mathrm{rank}(H)\leqslant\mathrm{rank}(G)-2$, then    $\mathrm{rank}(H)\geqslant\mathrm{rank}(G)-3$ and the properties (i){\bf--}(iv) also   hold.
\end{lem}

\begin{proof}{
If $H$ is  an induced  subgraph of $G$ with the  maximum  possible  order subject to    $\mathrm{rank}(H)<\mathrm{rank}(G)$, then
the  statements (i){\bf--}(iv) can be found among  the  results of \cite{kot} and also  \cite{lov}. In order to prove  the rest of the  assertion, we let $H$ be an induced  subgraph of $G$ with the  maximum  possible  order subject to  $\mathrm{rank}(H)\leqslant\mathrm{rank}(G)-2$. We first establish  that  $\mathrm{rank}(H)\geqslant\mathrm{rank}(G)-3$. Assume that
$H_1$  is an induced  subgraph of $G$  with the  maximum  possible  order subject to    $\mathrm{rank}(H_1)<\mathrm{rank}(G)$.
If $\mathrm{rank}(H_1)=\mathrm{rank}(G)-2$, then we clearly have  $\mathrm{rank}(H)=\mathrm{rank}(H_1)$.
Also, if $\mathrm{rank}(H_1)=\mathrm{rank}(G)-1$, then by the first part of the lemma,   $H_1$ is reduced and so $\mathrm{rank}(H_2)\geqslant\mathrm{rank}(G)-3$,  where $H_2$  is an induced  subgraph of $H_1$  with the  maximum  possible  order subject to    $\mathrm{rank}(H_2)<\mathrm{rank}(H_1)$.
It follows  that $\mathrm{rank}(H)\geqslant\mathrm{rank}(G)-3$.
By the definition of $H$ and by Lemma \ref{jad},  (i) and  hence (ii) is  valid . For (iii),  let  $H$ have  a  duplication  class containing three distinct vertices $x, y, z$. Clearly,   for every vertex $t\in V(G)\setminus V(H)$, at least one of the three symmetric differences of  $N(x), N(y), N(z)$ does not contain $t$. This  contradicts (i). The second statement  of (iii) follows from (ii), since $G$ is reduced.
For (iv),  note first that,   by the definition of $H$, any vertex in   $V(G)\setminus V(H)$ is  adjacent to exactly one vertex in each duplication class, since for any duplication class $\{x, y\}$ in $H$, we have $N(x)\triangle N(y)\subseteq H$.
If (iv) does not hold, then
$A(G)$ contains
\begin{equation}\label{abo}\left[
\begin{array}{c|c}
\begin{array}{c|c|c|c|c}
&  &  &  &  \\
&  &  &  &  \\
\bmi{x} & \bmi{x} & \bmi{y} &  \bmi{y}& \bmi{\star}  \\
&  &  &  &  \\
&  &  &  &
\end{array}
&
\begin{array}{cc} 1&1\\0&0\\1&0\\0&1\\ \star & \star \vspace{1mm}\end{array}
\\\hline
\begin{array}{ccccc} 1\,{} & 0\,{} & 1&{}\,0& {}\,\bmi{\star}\\1\,{} &  0\,{}  & 0&{}\,1& {}\,\bmi{\star}\end{array}
&
\begin{array}{cc}0& \star \\ \star&0 \end{array}
\end{array}
\right]
\end{equation}
as a principle submatrix, where  the     upper-left  corner of (\ref{abo}) is $A(H)$. This yields  that $\mathrm{rank}(H)\leqslant\mathrm{rank}(G)-4$, a contradiction.
}\end{proof}

For any
graph $G$, a subset $X$ of $V(G)$ is called  {\sl independent} if the induced subgraph
on $X$ has no edges. The maximum size of an independent set in a graph $G$ is
called the {\sl independence number} of $G$ and is denoted by $\alpha(G)$.
We will make   use of the following lemma which is an immediate consequence of     the Plotkin bound \cite[p.\,58]{huf}  from coding theory  and was  also  established   in \cite{gmt} by a direct proof.

\begin{lem}\label{in}
Let  $G$ be a  graph of order $n$ and let $S$ be an independent set  in $G$ with  $|S|\geqslant2$.  Then
 $$\min\big\{|N(u)\triangle N(v)| \, \big|\,  u, v \in S, u\neq v\big\}\leqslant\frac{|S|\big(n-|S|\big)}{2\big(|S|-1\big)}.$$
\end{lem}

In the following, we recall   the  Singleton bound \cite[p.\,71]{huf} from coding theory.

\begin{thm}\label{single}
Let $n$ be  a positive integer and  $\mathnormal{\Omega}$ be the set of all $(0, 1)$-vectors of length $n$.  Let  $C$ be a subset of $\mathnormal{\Omega}$ so that
every pair of the vectors in  $C$ differ in at least $d$ positions. Then
$|C|\leqslant2^{n-d+1}$.  The equality occurs if and only  if one of the following holds.
\begin{itemize}
\item [{\rm (i)}] $C=\mathnormal{\Omega}$.
\item [{\rm (ii)}] $C$ is  the set of all even weight vectors of $\mathnormal{\Omega}$.
\item [{\rm (iii)}] $C$ is   the set of all odd weight vectors of $\mathnormal{\Omega}$.
\item [{\rm (iv)}] $C$  consists of two vectors which are different in all positions.
\end{itemize}
\end{thm}

We will use   ${\bmi j}$  for   the all one vector.

\begin{lem}\label{f2n}
Let $C$ be a set of $(0, 1)$-vectors of length $n\geqslant5$ such that every  two distinct   vectors in $C$  differ in at least  $2$ positions.
Let $M$ be  the matrix  whose columns are the vectors in  $C$ and suppose that  ${\bmi j}$ is contained  in the row space of $M$. Then $|C|\leqslant5\cdot2^{n-4}$.
\end{lem}

\begin{proof}{
Toward a contradiction, suppose  that  $|C|>5\cdot2^{n-4}$. Let
\begin{equation}\label{xj}
(x_1,  \ldots,   x_n)M=\bmi{j},
\end{equation}
for some reals $x_1, \ldots, x_n$.
Let $M'$ be the matrix constituted from  the last $n-2$  rows of $M$  and partition the columns of $M'$ such that equal
columns belong  to the same part. Since  the number of parts in the partition is at most $2^{n-2}$ and  $5\cdot2^{n-4}>2^{n-2}$,  there is a  part  of size at least $2$.  Since every  two distinct   columns in $M$  differ in at least  $2$ positions,
we find  two columns in $M$
such that their entries   are the same at all positions except for the first and the second  positions. It follows from (\ref{xj}) that either
$x_1=x_2$ or $x_1=-x_2$.   By applying  this  argument to any pair of  rows of $M$ and  a suitable ordering of the rows of $M$, we find that $x_1=\cdots=x_k=-x_{k+1}=\cdots=-x_n$, for some $k$. Now, let $N$ be the matrix obtained from $M$ by subtracting ${\bmi j}$ from  $i$th row of $M$, for all $i\in\{k+1, \ldots, n\}$, and leaving the first $k$ rows intact.  We have ${\bmi j}N=(n-k+1/x_1){\bmi j}$. This means that  the column vectors  of $N$ have the same
number of ones  which in turn implies that $|C|\leqslant{n\choose\lfloor n/2\rfloor}$. This contradicts   $|C|>5\cdot2^{n-4}\geqslant{n\choose\lfloor n/2\rfloor}$, for $n\geqslant 5$.
}\end{proof}

It is an  interesting  problem  to determine the best upper  bound for $|C|$ in Lemma \ref{f2n}.

In \cite{gmt},  the maximum order of a reduced bipartite graph of rank $r$ is determined. The graph attaining the maximum order is unique and  is described as follows.
Let $B$ be a set of size $n$  and $\mathscr{B}$ be a family of subsets of $B$. The {\sl incidence graph} $(B, \mathscr{B})$ is the bipartite  graph  with bipartition $\{B, \mathscr{B}\}$ so that  the  vertices  $x\in B$ and $X\in\mathscr{B}$ are adjacent if and only if  $x\in X$.
If  $\mathscr{P}(B)$ is the family of all nonempty subsets of $B$, then we denote the incidence graph $(B, \mathscr{P}(B))$ by $\mathcal{B}_n$.
It is routine to verify that    $\mathcal{B}_n$ is a reduced bipartite graph of   rank $2n$ and   order $b(2n)$.
Further, we denote  by $\mathcal{O}_n$ the incidence graph corresponding to  the family of all  subsets of  $B$ of odd size.

\begin{thm}\label{bi} {\rm\cite{gmt}}
The order of a  reduced bipartite graph   of  rank $r$ is at most $b(r)=2^{r/2}+r/2-1$. Moreover, every reduced bipartite graph   of  rank $r$ and order  $ b(r)$ is isomorphic to $\mathcal{B}_{r/2}$.
\end{thm}

\section{Bipartite graphs}

For a bipartite graph $G$ with bipartition $\{X, Y\}$,
the submatrix of $A(G)$ whose  rows and  columns are respectively  indexed  by $X$ and $Y$ is called the   {\sl bipartite adjacency matrix} of $G$ and is denoted by $B(G)$.
To establish our main result,  we need the following theorem.  It is straightforward to see that it    generalizes   Theorem \ref{bi}.  We recall again  that the rank of a bipartite graph is always  even.

\begin{thm}\label{bigen}
Let $G$ be a reduced bipartite graph   of  rank $r\geqslant6$ and order
$n>c(r)=3\cdot2^{r/2-2}+r/2$ with bipartition $\{X, Y\}$.  Then  $\min\{|X|, |Y|\}=r/2$.
\end{thm}

\begin{proof}{ For simplicity, we set $\rho=r/2$.
We proceed by  induction on $\rho$. The assertion holds for $\rho=3$ by Theorem \ref{bi}. So assume that $\rho\geqslant4$.  It is clear that    $\mathrm{rank}(G)\leqslant2\min\{|X|, |Y|\}$ and hence    $\min\{|X|, |Y|\}\geqslant\rho$. Towards a contradiction, suppose that $\min\{|X|, |Y|\}\geqslant\rho+1$.

Let   $H$ be an induced  subgraph of $G$ with the  maximum  possible  order such that   $\mathrm{rank}(H)<\mathrm{rank}(G)$ and let $t=n-|V(H)|$. By Lemma \ref{lov} and since $H$ is bipartite,  $\mathrm{rank}(H)=r-2$.
In view of  Lemma \ref{lov}\,(iii),  suppose  that  $\{v_1, v_1'\}, \ldots, \{v_s, v_s'\}$ are     the duplication  classes of $H$, for some $s\geqslant0$, where the labeling of vertices comes from  Lemma \ref{lov}\,(iv).
For simplicity, set $S=\{v_1,\ldots,  v_s\}$ and $S'=\{v_1',\ldots,  v_s'\}$.
We denote the  number of isolated vertices of $H$  by $\epsilon$.
Lemma \ref{lov}\,(iii) implies that $\epsilon\in\{0, 1\}$.
Let  $K$ be  the  resulting  graph after deleting  the possible  isolated vertex    from   $H-S'$ and put  $k=|V(K)|$.
Clearly,    $\mathrm{rank}(K)=\mathrm{rank}(H)=r-2$ and   since $K$ is reduced,  $k\leqslant  b(r-2)$ by Theorem \ref{bi}.
Moreover, since  $\alpha(G)\geqslant n/2$,  Lemma \ref{lov}\,(i) and Lemma   \ref{in} imply   that  $t<(n+3)/4$.
It then follows from $n=k+s+t+\epsilon\geqslant c(r)+1$ and $k\leqslant  b(r-2)$ that $s>2^{\rho-4}-\rho/4+1$. This means that $s\geqslant2$.
Further, let  $T_1$ and $T_2$ be the sets  given  in Lemma \ref{lov}\,(iv).
We may assume that $V(G)\setminus V(H)\subseteq X$  and   $S\cup S'\subseteq Y$.
For this, assume with no loss of generality that  $T_1\cap X\neq\varnothing$ and let  $x\in T_1\cap X$. By Lemma \ref{lov}\,(iv), $x\in N(v_i)$, for  $i=1, \ldots, s$,  meaning that   $S\subseteq Y$.
Since any $v_i$ has some neighbor in $X\setminus T$ and $\{v_i,  v_i'\}$ is a duplication class in $H$,
we conclude that  $S'\subseteq Y$ and thus  $V(G)\setminus V(H)\subseteq X$.
Let  $P=Y\cap V(K-S)$,   $Q=X\cap V(K)$ and set   $p=|P|$,   $q=|Q|$. In  Figure \ref{bipF}, we depict   the structure of $G$ when $\epsilon=0$.
\begin{figure}[h]
\centering
\includegraphics[width=6cm]{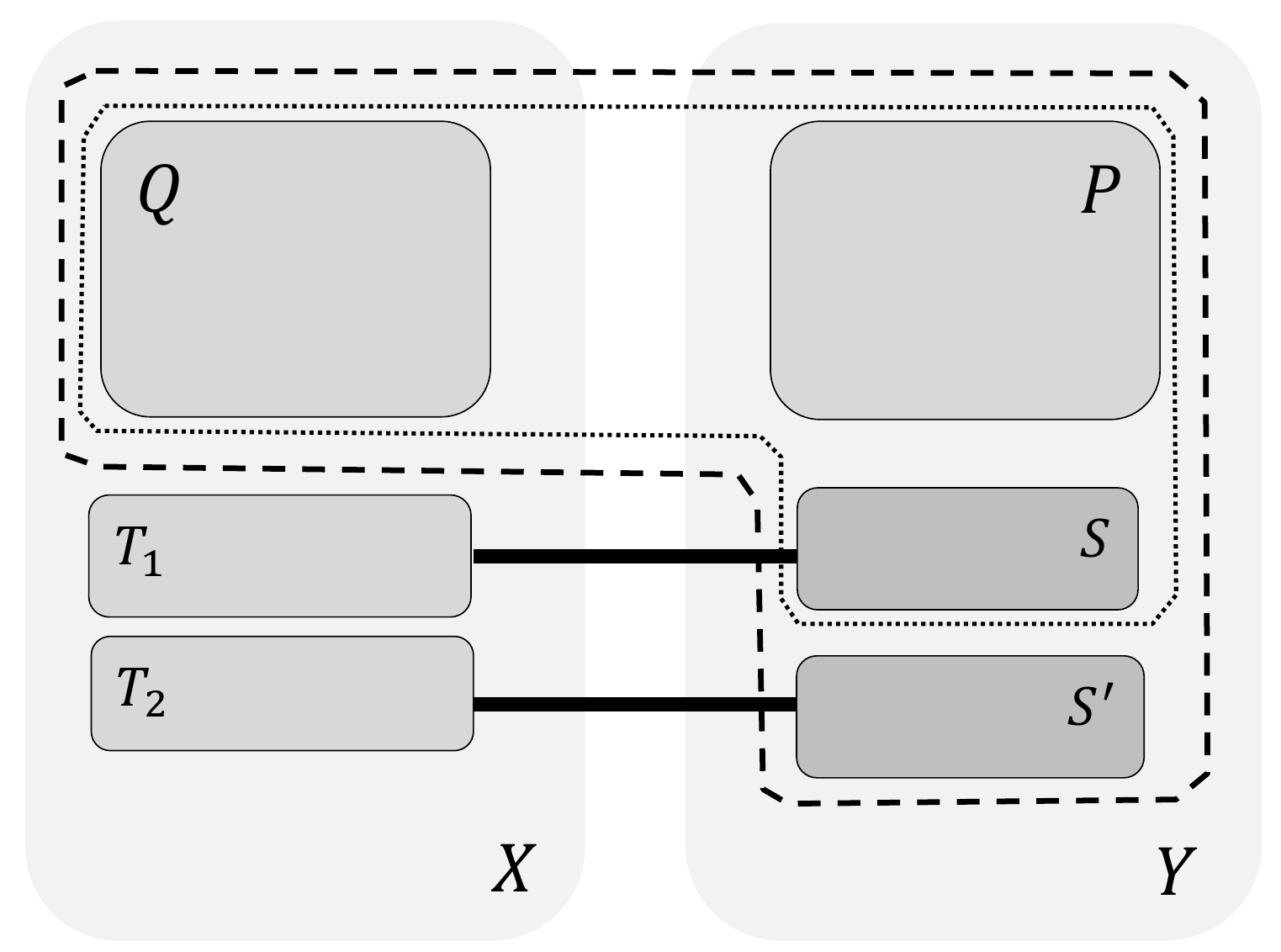}\\
\caption{The structure of $G$ concluded from Lemma \ref{lov} (The subgraphs $K\subset H$ are shown with dotted borders.)}\label{bipF}
\end{figure}

Since    $N(v_1)\triangle N(v_2)\subseteq Q$, Lemma \ref{lov}\,(i) yields   that
\begin{equation}\label{tq}
t\leqslant q.
\end{equation}
If   $t\geqslant3$,  then we may  assume with no loss of generality that $|T_1|\geqslant2$. By  Lemma \ref{lov}\,(iv),  $N(x)\triangle N(y)\subseteq P$,  for  two distinct vertices $x, y \in T_1$   and so  by   Lemma \ref{lov}\,(i), $t\leqslant p$. So, in general, we have  $t\leqslant p+2$. From  $n\geqslant c(r)+1$ and  $k\leqslant  b(r-2)$,  it follows that   $s+t=n-k-\epsilon\geqslant2^{\rho-2}+3-\epsilon$.
Since the symmetric difference of neighborhoods of any two vertices in $S$ is contained in $Q$ and has  size at least $t$
by  Lemma \ref{lov}\,(i), so  Theorem \ref{single} yields that
\begin{equation}\label{si}
s\leqslant 2^{q-t+1}
\end{equation}
and thus
\begin{eqnarray}\label{sin}
2^{\rho-2}+3-\epsilon\leqslant s+t\leqslant 2^{q-t+1}+t.
\end{eqnarray}
We claim  that  $t=2$  and $q=\rho-1$.  To establish the claim, we consider  the following two   cases.

\noindent{\bf{\textsf{Case 1.}}}  $k\leqslant(n+\rho-3)/2$.

From $n=k+s+t+\epsilon$  and $k=p+q+s$, we have  $p+q\leqslant t+\rho+\epsilon-3$.
If   $t\geqslant3$, then as we just showed,  $t\leqslant p$ and thus in view of (\ref{tq}), we have $t\leqslant q\leqslant\rho+\epsilon-3$. From  (\ref{sin}), we find that $2^{\rho-2}+2\leqslant 2^{\rho-4}+\rho-2$,  which is impossible.
Therefore     $t\leqslant2$. From  $p+q\leqslant t+\rho+\epsilon-3$
and    $q+t=|X|\geqslant\rho+1$, we obtain  that $\rho+1-t\leqslant q\leqslant\rho+t-2$ which in turn implies that  $t=2$ and either $q=\rho-1$ or $q=\rho$.
To get a contradiction,  assume that $q=\rho$.   Then $p+q\leqslant t+\rho+\epsilon-3$ yields that $\epsilon=1$  and  $p=0$. Since $P=\varnothing$,
if one of  $T_1$ or  $T_2$ is empty, then the other one will be a duplication class of $G$ by  Lemma \ref{lov}\,(iv).  Therefore   both $T_1$ and $T_2$  are nonempty,   since $G$ is reduced.
Hence we see   that
$$B(G)=\left[
\begin{array}{ccc}
B(K) & B(K) & \bmi{0} \\
\bmi{j} & {\bmi 0} &1 \\
{\bmi 0} &  \bmi{j}  & 1 \\
\end{array}
\right].$$
Since  $\mathrm{rank}(B(K))=\mathrm{rank}(K)/2=\rho-1$, one can easily check that  the rank of the  row space of $B(G)$ is $\rho+1$   which implies  that     $\mathrm{rank}(G)=r+2$,  a contradiction.
Therefore we must have      $q=\rho-1$, as claimed.

\noindent{\bf{\textsf{Case 2.}}}   $k>(n+\rho-3)/2$.

Since $n\geqslant c(r)+1$, we have  $k>c(r-2)$. By the induction hypothesis, $\min\{p+s, q\}=\rho-1$. If  $p+s=\rho-1$, then from  $p+2\geqslant t$,  we find  that  $$\rho-1=p+s=n-k+p-t-\epsilon\geqslant c(r)+1-b(r-2)-2-\epsilon\geqslant2^{\rho-2}$$ which is a contradiction  to   $\rho\geqslant4$. Hence  $q=\rho-1$. Since  $q+t=|X|\geqslant\rho+1$, we deduce that $t\geqslant2$. By (\ref{tq}),   $t\leqslant \rho-1$ and using  (\ref{sin}),  a straightforward calculation shows  that $t=2$,  as claimed.

As we proved that $t=2$  and $q=\rho-1$, it follows from  (\ref{sin}) that   $\epsilon=1$, implying that
the equality occurs  in  (\ref{si}). This means that the equality occurs in Theorem \ref{single} for $n=\rho-1$ and $d=2$.  Since  $K$ has no isolated vertex and $\rho\geqslant4$, the cases (ii) and (iv) do not occur and so
the induced subgraph on $Q\cup S$ is isomorphic to $\mathcal{O}_{\rho-1}$. If  both  $T_1$ and $T_2$ are nonempty,  then
$B(G)$ is of the form
$$\left[
\begin{array}{cccc}
B({\cal O}_{\rho-1}) & B({\cal O}_{\rho-1}) & \bmi{0} & \bmi{\star} \\
\bmi{j} & {\bmi 0} &1 & \bmi{\star} \\
{\bmi 0} &  \bmi{j}  & 1 & \bmi{\star}\\
\end{array}
\right].$$
Since  $\mathrm{rank}(B({\cal O}_{\rho-1}))=\rho-1$, we find   that  $\mathrm{rank}(G)\geqslant r+2$, a contradiction. So we may assume that $T_2$ is empty. Since the induced subgraph on $Q\cup S$ is isomorphic to $\mathcal{O}_{\rho-1}$, there exists a vertex  $v\in S$ such that $|N(v)\cap Q|=1$.
If $u$ is the isolated vertex of $H$, then  $|N(u)\triangle N(v)|=1$ which is impossible   by
Lemma \ref{lov}\,(i).
This contradiction completes the proof.
}\end{proof}

\section{Triangle-free graphs}

In this section, we establish  that every reduced non-bipartite  triangle-free graph  of rank $r$  has at most $ c(r)$ vertices.  We also prove that there exists a unique reduced non-bipartite triangle-free graph  of rank $r$ and order $c(r)$.

\begin{deff} For any integer  $r\geqslant4$, consider the graph ${\cal B}_{\lfloor r/2\rfloor-1}$ with bipartition $\{B,\mathscr{P}(B)\}$ and let $x\in B$. Let $N=N(x)$ and $M=\mathscr{P}(B)\setminus N$. For even $r$,  we duplicate $x$ and $M$ to produce $x'$ and $M'$.  Now, introduce  two new vertices $y,z$ and join $y$ to all vertices in $\{x, z\}\cup M$.
For odd $r$, duplicate $N$ and call it $N'$. Then introduce two new vertices $y,z$,   join $y$ to all vertices in  $\{z\}\cup N$ and join  $z$ to all vertices in $N'$. We denote the resulting graph by ${\cal C}_r$. Clearly, the order of  ${\cal C}_r$ is $c(r)$.  The graphs ${\cal C}_8$ and ${\cal C}_9$ are depicted  in Figure \ref{fig}.
\begin{figure}[h]
\centering
\includegraphics[width=10cm]{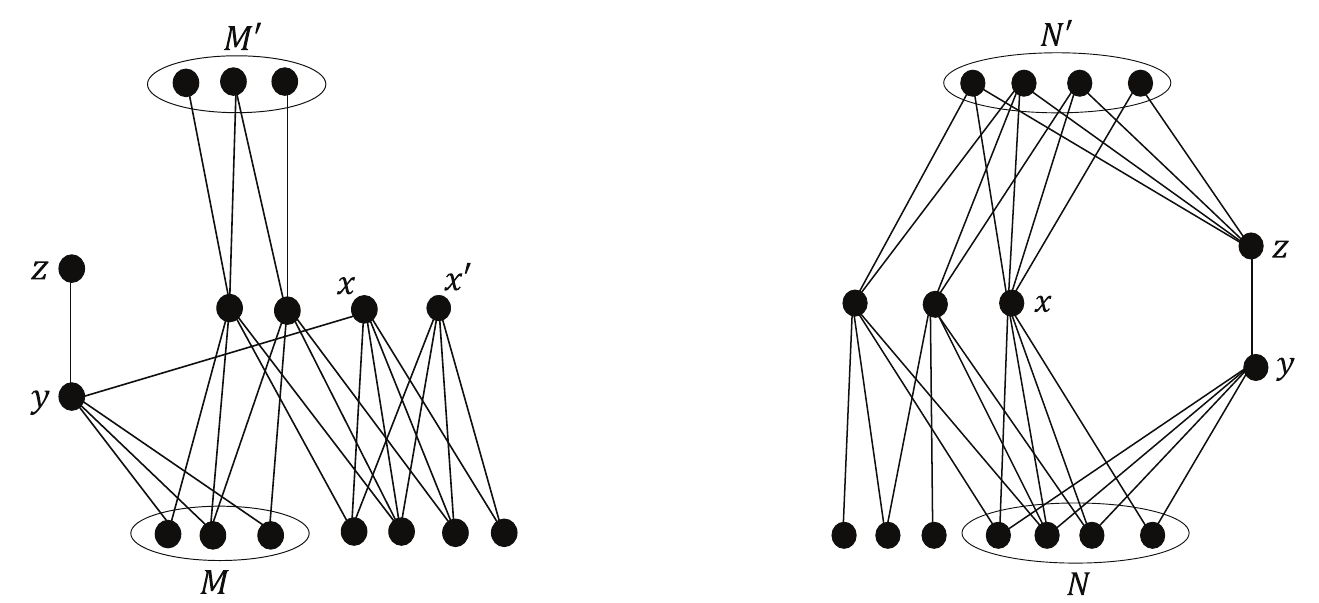}\\
\caption{The graphs ${\cal C}_8$ (left) and ${\cal C}_9$ (right)}\label{fig}
\end{figure}
\end{deff}

It is not hard to verify that one can define the graphs ${\cal C}_r$   recursively   as follows.    Let $\mathcal{C}_4$ and $\mathcal{C}_5$ be the path and the cycle on 5 vertices, respectively. For  $r=6$ (respectively, $r=7$),  let   $A$ be a set of two vertices of distance $3$ (respectively,  $2$) from each other in ${\cal C}_{4}$ (respectively, ${\cal C}_{5}$) and
for  $r\geqslant8$, let $A$ be the  maximum  independent set  of ${\cal C}_{r-2}$.
Now, duplicate  each vertex  in  $A$, introduce  two  new vertices $u, v$ and join $u$ to all vertices in $\{v\}\cup A$.

By the inductive definition of $\mathcal{C}_r$, it is easily seen that
${\cal C}_r$ is a reduced non-bipartite triangle-free graph and $\mathrm{rank}(\mathcal{C}_r)=\mathrm{rank}(\mathcal{C}_{r-2})+2$.
It follows that    $\mathrm{rank}(\mathcal{C}_r)=r$,  for $r\geqslant4$.
Furthermore,  we easily find from the definition of $\mathcal{C}_r$ that
\begin{equation}\label{alphaC}
\alpha({\cal C}_r)=3\cdot2^{\lfloor r/2\rfloor-2}-1,
\end{equation}
for  $r\geqslant6$,
and  ${\cal C}_r$ has a unique    independent set of size $\alpha({\cal C}_r)$.

\begin{remark}
Note  that  in Theorem \ref{bigen}, the hypothesis that  $n>c(r)$  cannot be weakened. For any odd  $r\geqslant7$, if one removes the edge $\{y, z\}$  of  $\mathcal{C}_{r}$, then the resulting graph, say  $H$,  is a reduced bipartite graph.
Consider the graph $H-\{z\}$. Removing $y$ from that results in the graph ${\cal B}_{(r-3)/2}$ with the neighborhood of $x$ duplicated.
So, ${\rm rank}(H-\{y,z\})=r-3$ and clearly ${\rm rank}(H-\{z\})\leqslant r-1$. Since $H-\{z\}$ is reduced,  we must have from Lemma \ref{jad}\,(ii) that  ${\rm rank}(H-\{y, z\})\leqslant{\rm rank}(H-\{z\})-2$ and so ${\rm rank}(H-\{z\})=r-1$. The sum of the row vectors  corresponding to $z$ and $y$ in $A(H)$ is equal to that of $x$, so ${\rm rank}(H)={\rm rank}(H-\{z\})$. Therefore,  $H$ is a reduced bipartite graph
with bipartition $\{X, Y\}$ of rank $r-1$ and order $c(r-1)$ where $\min\{|X|, |Y|\}=(r+1)/2$.
\end{remark}

\begin{thm}\label{main}
The order of a  reduced non-bipartite triangle-free graph   of  rank $r$ is at most $c(r)=3\cdot2^{\lfloor r/2\rfloor-2}+\lfloor r/2\rfloor$. Moreover, every reduced non-bipartite triangle-free graph   of  rank $r$ and order  $c(r)$ is isomorphic to $\mathcal{C}_r$.
\end{thm}

\begin{proof}{
Let  $G$ be a  reduced non-bipartite triangle-free graph of  rank $r$  and order $n\geqslant c(r)$.
By induction on $r$, we  prove that  $G$ is isomorphic to $\mathcal{C}_r$.
In  \cite{akb, ell},  an algorithm is given to construct   all reduced graphs of a given rank. We employed the  algorithm  and
verified  that  the assertion holds  for $r\leqslant9$.  The source code of our program can be found at
{\tt http://math.ipm.ac.ir/\textasciitilde tayfeh-r/Trianglefree.htm}.
Hence let $r\geqslant10$. For simplicity, we set $\rho=\lfloor r/2\rfloor$.
Let $T$ be a subset of $V(G)$ with the minimum possible size  such that $\mathrm{rank}(G-T)\leqslant\mathrm{rank}(G)-2$. Put $H=G-T$ and  $t=|T|$.
We  show that $t<(n+3)/3$. If the minimum degree of $G$ is less than  $(n+3)/3$, then we are done by
Lemma \ref{lov}\,(i). Otherwise,  since $G$ is  triangle-free,  $\alpha(G)\geqslant(n+3)/3$ and   by  Lemma \ref{lov}\,(i) and Lemma \ref{in}, we have
$$t\leqslant \frac{\frac{n+3}{3}\left(n-\frac{n+3}{3}\right)}{2\left(\frac{n+3}{3}-1\right)}<\frac{n+3}{3},$$ as required.
In view of  Lemma \ref{lov}\,(iii),  suppose  that  $\{v_1, v_1'\}, \ldots, \{v_s, v_s'\}$ are     the duplication  classes of $H$, for some  $s\geqslant0$, where the labeling of vertices comes from  Lemma \ref{lov}\,(iv).
For simplicity, put  $S=\{v_1,\ldots,  v_s\}$ and $S'=\{v_1',\ldots,  v_s'\}$.
Since $G$ is triangle-free, by Lemma \ref{lov}\,(iv),    $S\cup S'$ is an independent set.
Denote the  number of isolated vertices of $H$  by $\epsilon$. By   Lemma \ref{lov}\,(iii),   $\epsilon\in\{0, 1\}$.
Let $K$ be  the  resulting  graph after deleting  the possible  isolated vertices    from   $H-S'$ and set  $k=|V(K)|$. By   Lemma \ref{lov}, we have $\mathrm{rank}(K)\geqslant r-3$.
Set  $P=V(K)\setminus S$ and  $p=|P|$.  Further, let  $T_1$ and $T_2$ be the sets  given  in Lemma \ref{lov}\,(iv)  with sizes $t_1$ and $t_2$, respectively. With no loss of generality, we   assume that
$t_1\geqslant t_2$.
We consider the following two cases.

\noindent{\bf{\textsf{Case 1.}}}  $k\leqslant c(r-2)$.

Let $P_1$ be the set of vertices in  $P$ which have a neighbor in  $S$. Set  $p_1=|P_1|$ and  $p_2=|P\setminus P_1|$. For the structure of $G$ when $\epsilon=0$, see Figure \ref{FigCase1}.
\begin{figure}[h]
\centering
\includegraphics[width=6cm]{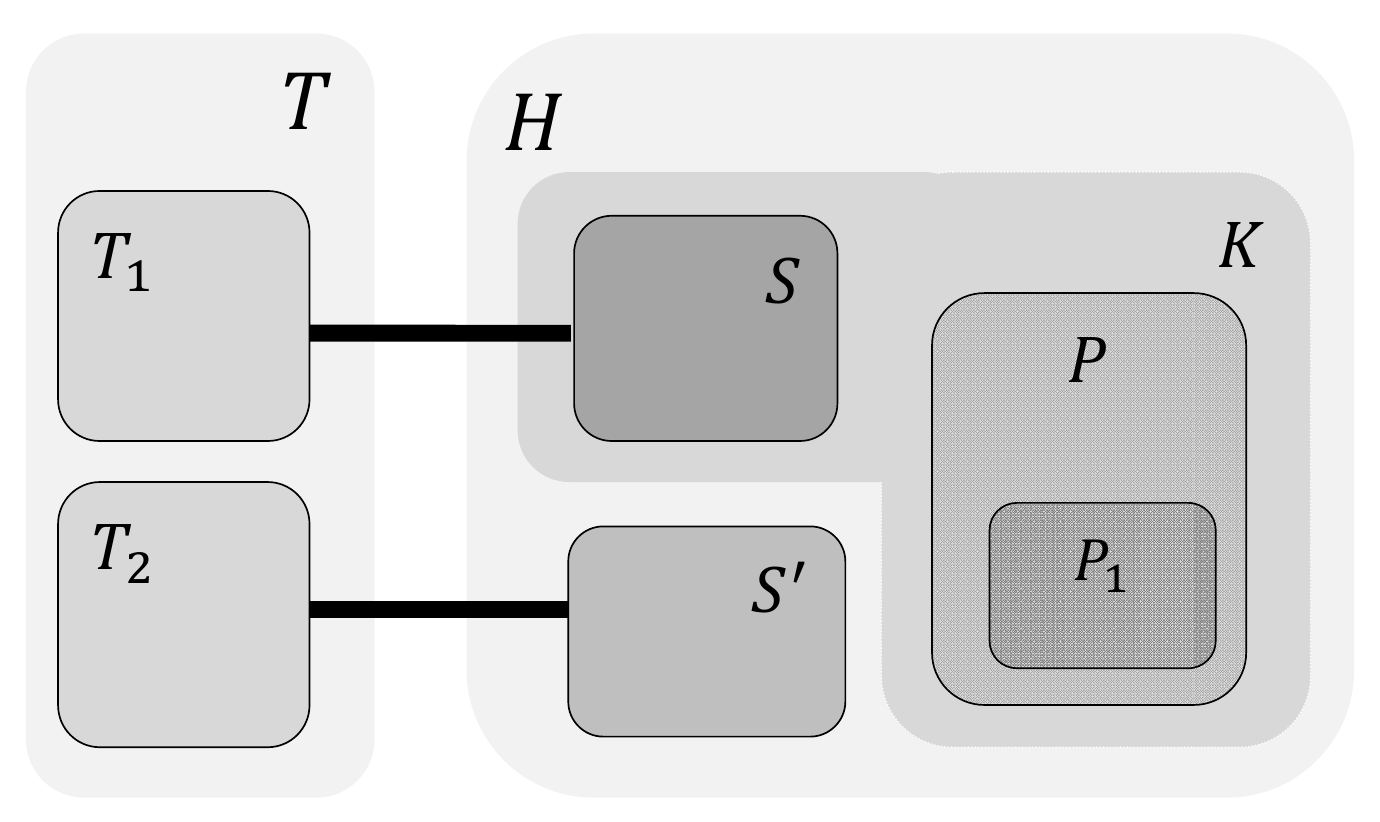}\\
\caption{The structure of $G$ in Case 1}\label{FigCase1}
\end{figure}
Since $G$ is triangle-free, there is no edge between $P_1$ and $T_1$.
We have
\begin{eqnarray}
\label{st}  s+t=n-k-\epsilon \geqslant 3\cdot2^{\rho-3}+1-\epsilon.
\end{eqnarray}
From  $k=p+s\leqslant c(r-2)$ and (\ref{st}), we see that
$$3\cdot2^{\rho-3}+1-\epsilon-t\leqslant s\leqslant 3\cdot2^{\rho-3}+\rho-1-p$$
which implies that
\begin{equation}\label{p}
p\leqslant t+\rho+\epsilon-2.
\end{equation}
From $t<(n+3)/3$, $n\geqslant c(r)$ and $k\leqslant c(r-2)$, we find that  $s=n-k-t-\epsilon>2^{\rho-3}-\rho/3-1$ and so $s\geqslant2$. Now,  since  $N(v_1)\triangle N(v_2)\subseteq P_1$,
\begin{equation}\label{t<p}
t\leqslant p_1.
\end{equation}
By (\ref{st}), Lemma \ref{lov}\,(i) and  Theorem \ref{single}, we have
\begin{equation}\label{<st<}
3\cdot2^{\rho-3}+1-\epsilon\leqslant s+t\leqslant2^{p_1-t+1}+t.
\end{equation}

Towards   a contradiction,
suppose that $t_1\geqslant2$. Then   $t\leqslant|N(u)\triangle N(v)|\leqslant t_2+p_2$,  for each  pair  $u, v\in T_1$, and thus  $t_1\leqslant p_2$.
If $\epsilon=0$, then by (\ref{p}) and (\ref{t<p}), $t/2\leqslant t_1\leqslant p_2\leqslant\rho-2$ and hence  $t\leqslant2\rho-4$. Moreover, it follows from  (\ref{p}) and $2\leqslant t_1\leqslant p_2$ that   $p_1\leqslant t+\rho-4$. From   (\ref{<st<}), we conclude  that
$3\cdot2^{\rho-3}+1\leqslant2^{\rho-3}+2\rho-4$,   a contradiction.
Therefore  $\epsilon=1$. By Lemma \ref{lov}\,(ii), $N(u)\triangle N(v)\subseteq P\setminus P_1$, for any vertices $u, v\in T_1$, and hence $t\leqslant p_2$.  Also, it follows from  (\ref{p}) and $t\leqslant p_2$ that   $p_1\leqslant\rho-1$. Combining this with (\ref{t<p}) gives $t\leqslant\rho-1$, while combining with (\ref{<st<}) gives  $3\cdot2^{\rho-3}\leqslant2^{\rho-t}+t$ which is a contradiction to $\rho\geqslant5$. Thus  $t_1=1$ and so  $t_2\leqslant1$. Now we have
\begin{equation}\label{g}
p_1\geqslant\rho-1,
\end{equation}
since if $p_1\leqslant\rho-2$, then by (\ref{<st<}), $3\cdot2^{\rho-3}\leqslant2^{p_1}+2\leqslant2^{\rho-2}+2$
which is impossible for $\rho\geqslant5$.
We proceed to show that $t_2=0$. For this, we first  establish the following  property of $K$.

We show that if $K$ is a bipartite graph with bipartition $\{K_1, K_2\}$, then $S$ is contained in one  of $K_1$ or $K_2$.  With no loss of generality, assume that   $\ell=|P_1\cap K_1|\leqslant p_1/2$.  If $\ell=0$, then $P_1\subseteq K_2$, so that every vertex in $S$, begin adjacent to a vertex in $P_1$, must be in $K_1$.   Suppose $\ell\geqslant1$.  In order to get a contradiction, we first claim  that $\ell=1$. By Theorem \ref{single}, we obtain that
\begin{equation}\label{2sin}
|K_1\cap S|\leqslant 2^{p_1-\ell-t+1} \quad\text{and}\quad |K_2\cap S|\leqslant2^{\ell-t+1}.
\end{equation}
By (\ref{p}), $p_1\leqslant t-p_2+\rho-1$ and so $\ell\leqslant(\rho+1)/2$. Using (\ref{st}), (\ref{p}) and (\ref{2sin}),  we find that
\begin{equation}\label{ezaf}
3\cdot2^{\rho-3}-t\leqslant s\leqslant 2^{p_1-\ell-t+1}+2^{\ell-t+1}\leqslant 2^{\rho-p_2-\ell}+2^\ell.
\end{equation}
If $p_2+\ell\geqslant3$ and $\rho\geqslant6$, then by (\ref{ezaf}),  $$3\cdot2^{\rho-3}-2\leqslant2^{\rho-p_2-\ell}+2^\ell\leqslant2^{\rho-3}+2^{(\rho+1)/2}<2^{\rho-3}+2^{\rho-2}-2,$$ a contradiction.  If  $p_2+\ell=2$ and $\rho\geqslant6$, then $3\cdot2^{\rho-3}-2\leqslant2^{\rho-p_2-\ell}+2^\ell\leqslant2^{\rho-2}+4$ which is again impossible.  This implies that  if $\rho\geqslant6$, then  $\ell=1$. Now,  assume that $\rho=5$. By (\ref{ezaf}), we have
\begin{equation}\label{rho5}
12-t\leqslant s\leqslant 2^{p_1-\ell-t+1}+2^{\ell-t+1}.
\end{equation}
Meanwhile, (\ref{p}) gives
\begin{equation}\label{p1p2}
p_1+p_2\leqslant t+4.
\end{equation}
If $\ell\geqslant3$, then by  $\ell\leqslant p_1/2$, $t\leqslant2$ and (\ref{p1p2}),
we have $p_1=6$, $p_2=0$ and $t=2$ which violate (\ref{rho5}).
Hence $\ell\leqslant2$.
If $\ell=2$ and $t=1$,  then  by (\ref{rho5}) and (\ref{p1p2}), we see $p_1=5$, $p_2=0$ and
the equality occurs in one of the    inequalities  of (\ref{2sin}). By Theorem \ref{single}\,(i), $K$ has an isolated vertex, a contradiction.
Further, if $\ell=t=2$, then  by (\ref{rho5}) and (\ref{p1p2}), we have $p_1=6$, $p_2=0$ and
the equality occurs in both   of the inequalities  of (\ref{2sin}). Since $K$ is reduced, from Theorem \ref{single}\,(iii), one can  deduce that the resulting graph after deleting all edges whose   endpoints are in $P_1$
is isomorphic to the disjoint union of $\mathcal{O}_2$ and $\mathcal{O}_4$. Since $\mathcal{O}_2$ is disjoint union of two edges and $\mathrm{rank}(\mathcal{O}_4)=8$,  it is easily seen that $\mathrm{rank}(K)\geqslant12$ which contradicts $\mathrm{rank}(K)\leqslant r-2\leqslant9$.
So we conclude  that $\ell=1$ and this completes the proof of the claim.
Note that for any vertex  $u\in K_2\cap S$, we have  $N(u)\cap V(K)\subseteq P_1\cap K_1$. Since $K$ is reduced and $\ell=|P_1\cap K_1|=1$,
it follows that $K_2\cap S$ has one element, say $y$. Letting   $\{x\}=P_1\cap K_1$,   every  duplication class of   $K-\{x, y\}$ is  contained in $K_2$,
since $K$ is reduced and $N(y)\cap V(K)=\{x\}$.
Also,  every  duplication class of  $K-\{x, y\}$ has at most two elements,  since otherwise,  if $u_1, u_2, u_3$ belong to a duplication class, then at least two of them  would  be  duplicates  in $K$, a contradiction.
If  $K'$  is the reduced graph corresponding to  $K-\{x, y\}$, then by Lemma \ref{jad}\,(i) we obtain that ${\rm rank}(K')\leqslant{\rm rank}(K)-2\leqslant r-4$. Since ${\rm rank}(K')$ is even,  ${\rm rank}(K')\leqslant2\rho-4$.
By (\ref{g}), we have $|P_1\cap V(K')|\geqslant(\rho-2)/2$ and  therefore,  using (\ref{st}) and
Theorem \ref{bi}, we obtain that $(\rho-2)/2+3\cdot2^{\rho-3}-3\leqslant|V(K')|\leqslant b(2\rho-4)$  which  is  a contradiction to $\rho\geqslant5$. This establishes the desired property of $K$.

Working towards   a contradiction,
suppose that  $t_2=1$.  By (\ref{<st<}),
\begin{equation}\label{rhoep}
    3\cdot2^{\rho-3}-1-\epsilon\leqslant s\leqslant 2^{p_1-1}
\end{equation}
which yields that $\rho\leqslant p_1$. Meanwhile, by (\ref{p}), we have $p_1\leqslant\rho+\epsilon-p_2$. It follows that either $p_1=\rho$ or  $p_1=\rho+1$.
First, assume that   $p_1=\rho$.
The matrix $A(G)$ contains
\begin{equation}\label{matrix}
\bbordermatrix{~ & S & S'&P_1&P_2&T_1&T_2 \cr
S&  {\bf0} & {\bf0} & B^\top  &{\bf0}  & \bmi{j}^\top & {\bf0} \cr
S'&{\bf0} & {\bf0} & B^\top  & {\bf0} & {\bf0} & \bmi{j}^\top \cr
P_1&B & B & \bmi{\star} & \bmi{\star} & \bmi{\star} & \bmi{\star} \cr
P_2&{\bf0} & {\bf 0} & \bmi{\star} & \bmi{\star}  & \bmi{\star}  & \bmi{\star} \cr
T_1&\bmi{j} & {\bf 0} & \bmi{\star} &\bmi{\star} & \star & \star \cr
T_2&{\bf0} & \bmi{j} & \bmi{\star} & \bmi{\star} & \star & \star \cr}
\end{equation}
as a principal submatrix.
Since $\mathrm{rank}(K)\geqslant r-3$, the upper-left $4\times4$ block submatrix of (\ref{matrix}) has rank at least $r-3$.
If $\bmi{j}$ is not contained in the row space of $B$, then the rank of (\ref{matrix}) would be at least $r+1$, a contradiction.
Now, applying  Lemma \ref{f2n} to the  column vectors of $B$, we find that $s\leqslant5\cdot2^{\rho-4}$.  If $\rho\geqslant6$,  this is less that $3\cdot2^{\rho-3}-3$, contradicting (\ref{rhoep}).
If $\rho=5$, then $r\geqslant10$, $s=10$, $\epsilon=1$. Hence, $2s+p_1+p_2+\epsilon+t=n\geqslant c(r)\geqslant c(10)=29$ and $p_2\leqslant\epsilon$. This gives $p_2=1$ and $k=s+p_1+p_2=16=c(8)=c(9)$.
Thus $K$ is isomorphic to either ${\cal C}_8$ or ${\cal C}_9$.
However, $K$ contains the  independent set $S$ of size $10$ in which  $|N(u)\triangle N(v)|\geqslant2$ for every distinct $u,v\in S$ while neither ${\cal C}_8$ nor ${\cal C}_9$
has such an independent set.
Therefore $p_1=\rho+1$, $p_2=0$ and $\epsilon=1$.
Note that from  (\ref{st}) and  $k=s+p_1\leqslant c(r-2)$, we have  $s=3\cdot2^{\rho-3}-2$ and thus    $k=c(r-2)$.
By the preceding paragraph,  $K$ is not bipartite, since otherwise $G$ would be bipartite.
Applying the induction hypothesis,  $K$ is isomorphic to either  $\mathcal{C}_{r-2}$ if $\mathrm{rank}(K)=r-2$ or $\mathcal{C}_{r-3}$ if $r$ is odd and $\mathrm{rank}(K)=r-3$.
Hence, in view of (\ref{alphaC}), $S$ is  a maximal  independent  set of size $\alpha(K)-1$ in $K$. To arrive at a contradiction,
we show that   ${\cal C}_m$ has no  maximal  independent  set of size $\alpha({\cal C}_m)-1$, for every integer $m\geqslant8$.  This can be  directly checked when $m=8$ or $m=9$. For $m\geqslant10$, we  see that  the degree of  any vertex  of  $\mathcal{C}_m$  not contained in the unique maximum independent set is  at least  $2^{\lfloor m/2\rfloor-2}$. Thus every  independent set not contained in the unique maximum independent set is of size   at most
$c(m)-2^{\lfloor m/2\rfloor-2}<\alpha({\cal C}_m)-1$. Therefore every independent set of size $\alpha({\cal C}_m)-1$ in  ${\cal C}_m$ is contained in the unique  maximum independent set which  means that    ${\cal C}_m$ has no maximal  independent  set of size $\alpha({\cal C}_m)-1$, as desired.

Therefore  $t_2=0$. Again $K$ is not bipartite, since otherwise $G$ would be bipartite.   It follows from (\ref{g}) and  $k=s+p_1+p_2\leqslant c(r-2)$ that
$s\leqslant 3\cdot2^{\rho-3}-p_2$.
If  $s=3\cdot2^{\rho-3}$, then  $p_2=0$, requiring  that    $k=c(r-2)$.
By the induction hypothesis,  $K$ is isomorphic to either  $\mathcal{C}_{r-2}$ or $\mathcal{C}_{r-3}$ and so    $\alpha(K)=3\cdot2^{\rho-3}-1$ which contradicts  $s=3\cdot2^{\rho-3}$.
Hence   (\ref{st}) yields  that  $s=3\cdot2^{\rho-3}-1$ and   $\epsilon=1$. Then   from  $n\geqslant c(r)$, we have $p\geqslant\rho$ which in turn by  (\ref{p}) gives $p=\rho$ and so $k=s+p=c(r-2)$.
By the induction hypothesis,  $K$ is isomorphic to either  $\mathcal{C}_{r-2}$ or $\mathcal{C}_{r-3}$ and so    $\alpha(K)=3\cdot2^{\rho-3}-1$.
This implies that $p_2=0$ and so $p_1=\rho$. Now, the
inductive definition of $\mathcal{C}_r$ shows  that $G$ is isomorphic to $\mathcal{C}_r$.

\noindent{\bf{\textsf{Case 2.}}} $k>c(r-2)$.

By the  induction hypothesis,   $K$ is a bipartite graph  with  bipartition, say
$\{P_1\cup S_1, P_2\cup S_2\}$,  where   $P=P_1\cup P_2$ and  $S=S_1\cup S_2$.
Set $p_i=|P_i|$  and $s_i=|S_i|$,  for $i=1, 2$.
With no loss of generality, we may assume that $s_1+p_1\leqslant s_2+p_2$. Since $\mathrm{rank}(K)=2\rho-2$,  Theorem \ref{bigen} implies that
\begin{equation}\label{s1p1}
s_1+p_1=\rho-1.
\end{equation}
Let $S_i'=\{v_j'\, |\,v_j\in S_i, 1\leqslant j\leqslant s\}$  for $i=1, 2$. For the structure of $G$ when $\epsilon=0$, see Figure \ref{FigCase2}.
\begin{figure}[h]
\centering
\includegraphics[width=6cm]{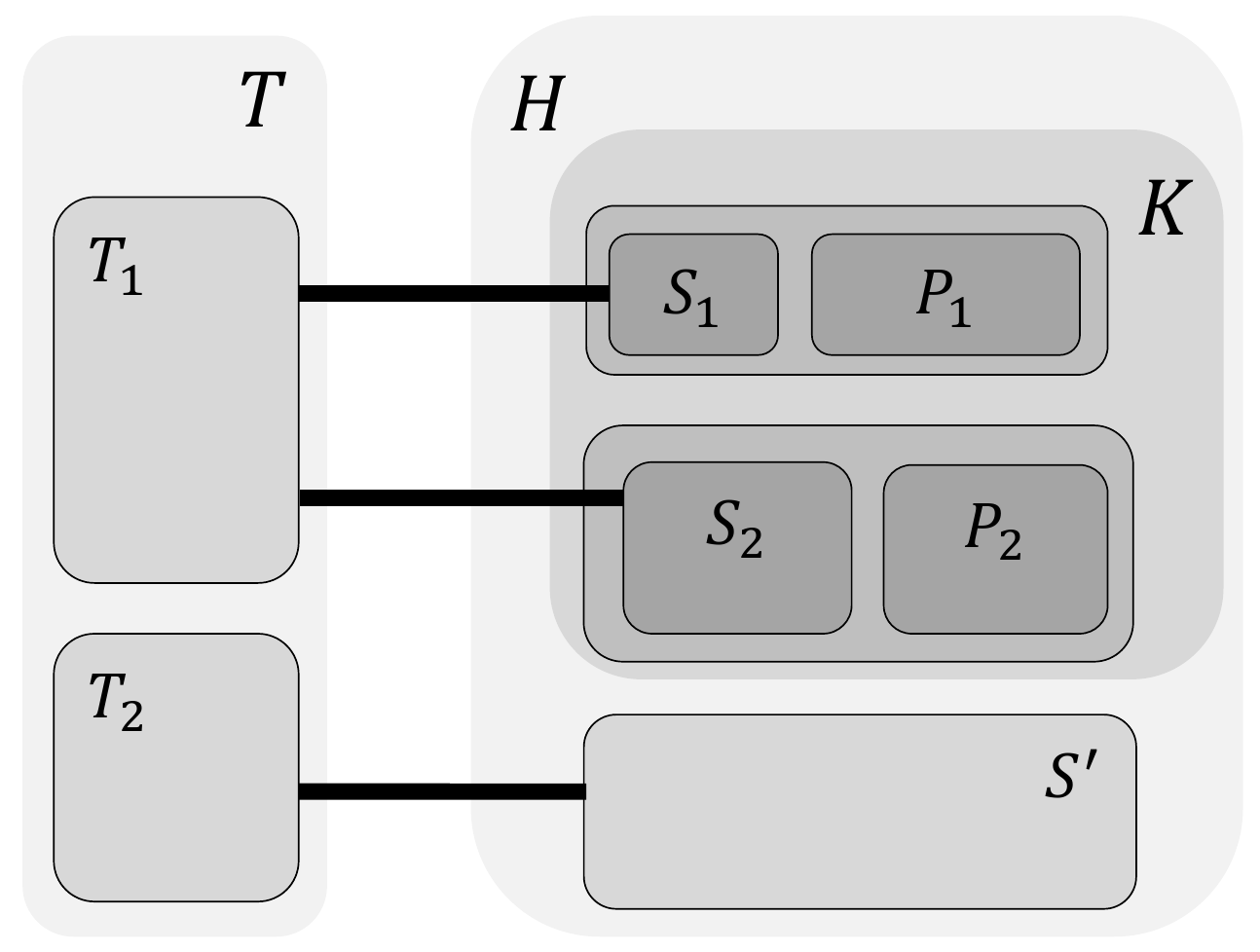}\\
\caption{The structure of $G$ in Case 2}\label{FigCase2}
\end{figure}

Working towards   a contradiction,
suppose that    $s_2\leqslant1$. We claim that $t\leqslant2\rho-2$.
Assume that  $s_1\geqslant1$.
Since $K$ is reduced,  there exists a vertex $u\in P_2$ with a neighbor in $S_1$. Since   $G$ is triangle-free, $N(u)\subseteq S_1\cup S_1'\cup P_1$ and so by Lemma \ref{lov}\,(i), we deduce that $t\leqslant2s_1+p_1\leqslant2\rho-2$, as desired.
Assume that $s_1=0$ and   $s_2=1$. It is easily  seen that the minimum degree among all vertices in  $S_2\cup S_2'$ does not exceed  $t_2+p_1$. By
Lemma \ref{lov}\,(i), we find that $t\leqslant t/2+\rho-1$ and so $t\leqslant2\rho-2$, as required.
Now, assume that $s_1=s_2=0$. From $t<(n+3)/3$, we find  that
\begin{align*}
\alpha(G)&\geqslant p_2+\epsilon\\&=n-t-p_1\\&> n-\left(\tfrac{n}{3}+1\right)-(\rho-1)\\&\geqslant 2^{\rho-1}-\tfrac{\rho}{3}.
\end{align*}
Therefore, $\alpha(G)\geqslant15$. From   $n-\alpha(G)\leqslant t+p_1=t+\rho-1$, Lemma \ref{lov}\,(i) and Lemma \ref{in}, we deduce that $t\leqslant\tfrac{15}{13}(\rho-1)$.
This establishes  the claim.
Now, by Theorem \ref{bi},   $$c(r)\leqslant n=k+t+s_1+s_2+\epsilon\leqslant b(2\rho-2)+3(\rho-1)+2$$ which  implies that  $\rho=5$ and so $s_2=n-k-s_1-t-\epsilon\geqslant10-s_1-t-\epsilon$. Hence
\begin{equation}\label{s1t1}
s_1+t+\epsilon\geqslant9.
\end{equation}
First assume that  $s_1=0$.  Since $t\leqslant2\rho-2=8$, we conclude that  $t=8$,  $\epsilon=1$ and  $s_2=1$.
By Lemma \ref{lov}, the vertices in $S_2\cup S_2'$ have  degree at least $8$.
On the other hand, the degree of any vertex of $S_2$ and $S'_2$ is at most $p_1+t_1$ and $p_1+t_2$, respectively.
By (\ref{s1p1}), $p_1=4$ and as $t_1+t_2=8$, we conclude that $t_1=t_2=4$ and  every vertex in  $P_1$ is adjacent to every vertex in $S_2\cup S_2'$.
This   shows that there is no edge between $P_1$ and  $T$ which in turn implies that   $G$ is bipartite, a contradiction.
Now,  suppose that $s_1\geqslant1$. If $p_1=0$, then  $s_2=0$ and so
there is no edge between $P_2$ and  $T$ which again implies that   $G$ is bipartite, a contradiction.
Hence $p_1\geqslant1$ and so by (\ref{s1p1}),  $s_1\leqslant3$. This, in view of (\ref{s1t1}), implies that $t\geqslant5$.
There is a vertex  $v\in P_2$ of degree $2$ in $K$  with a neighbor in $S_1$. To see this, note that $K$ is reduced and so we can view the vertices of $S_2\cup P_2$ as distinct nonempty subsets of $S_1\cup P_1$. If there does not exist such a vertex $v$, then $|S_2\cup P_2|\leqslant12$ implying that $k\leqslant16$ which is impossible as $k>c(8)=16$.
Since $v$ has a neighbor in $S_1$ and  $G$ is triangle-free, we deduce that  $N(v)\subseteq S_1\cup S_1'\cup P_1$. It follows from  Lemma \ref{lov}\,(i)  that $t\leqslant4$, contradicting (\ref{s1t1}). This contradiction establishes that $s_2\geqslant2$.

Since $n\geqslant c(r)$ and $k\leqslant b(2\rho-2)$, we obtain that
\begin{eqnarray}
\label{stb}  s+t=n-k-\epsilon \geqslant 2^{\rho-2}+2-\epsilon.
\end{eqnarray}
For any pair $u,v\in S_2$, we have $t\leqslant |N(u)\triangle N(v)|\leqslant p_1$. By  (\ref{stb}) and Theorem \ref{single},
\begin{align*}
2^{\rho-2}+2-\epsilon&\leqslant s+t\\&=\rho-1-p_1+s_2+t\\&\leqslant\rho-1+s_2\\
&\leqslant \rho-1+2^{p_1-t+1} \\
&=\rho-1+2^{\rho-s_1-t}.
\end{align*}
Since  $\rho\geqslant5$, we have  $s_1+t\leqslant2$.
Towards a contradiction,  assume that  $t=2$.
Then $s_1=0$, so that $p_1=\rho-1$ by (\ref{s1p1}). If some $v\in P_1$ has a neighbor in $T$, then, since $G$ is triangle-free, the neighborhood of each vertex in $S_2$ is a subset of $P_1\setminus\{v\}$ and hence has size at most $p_1-1=\rho-2$.
Thus  by Theorem \ref{single}, $s_2\leqslant2^{\rho-3}$ which
contradicts   (\ref{stb}). So there is no  edge between $T$ and $P_1$. Since $G$ is not bipartite,  there is an  edge with endpoints  in $T$.
Since $G$ is triangle-free, Lemma \ref{lov}\,(ii) implies that    $\epsilon=0$. From   (\ref{stb}) and Theorem \ref{single}, we obtain that   $s_2=2^{\rho-2}$.
Since $n\geqslant c(r)$ and $k\leqslant b(2\rho-2)$, we obtain that   $p_2=2^{\rho-2}-1$.
By Theorem \ref{single}, the neighborhoods of vertices of $P_2$ (respectively,  $S_2$) in $P_1$ correspond to odd-size  (respectively, even-size) subsets of $P_1$.
Let $T=\{a_1, a_2\}$. Since $G$ is triangle-free and there is an edge in $T$, we may assume that $T_1=\{a_1\}$ and $T_2=\{a_2\}$.
If $a_2$ is adjacent to a vertex  $x\in P_2$, then  Theorem \ref{single}\,(iii) implies that there exists a vertex $y\in S_2'$ such that $|N(x)\triangle N(y)|=1$ which is impossible by  Lemma \ref{lov}\,(i).  Therefore  $N(a_2)=S_2'$.
Now $G-(S_2'\cup\{a_2\})$ is a  bipartite graph with  bipartition  $\{P_1\cup\{a_1\}, S_2\cup P_2\}$, and is reduced by Theorem \ref{single}.
Since the number of vertices of $G-(S_2'\cup\{a_2\})$ is larger than $b(2\rho-2)$, Theorem \ref{bi} implies that $2\rho\leqslant\mathrm{rank}(G-(S_2'\cup\{a_2\}))$. On the other hand, by Lemma  \ref{jad}\,(i), $\mathrm{rank}(G-(S_2'\cup\{a_2\}))\leqslant r-2$. These give $2\rho\leqslant r-2$ which is impossible.

Therefore   $t=1$. From  $s_1+t\leqslant2$, we have  $s_1\leqslant1$. Suppose that   $s_1=0$. As $G$ is not bipartite, there must be an edge between $T$ and $P_1$. So there is a vertex in $P_1$ with no neighbor in $S_2$.  Now, Theorem \ref{single} and  (\ref{stb}) imply  that  $s_2=2^{\rho-2}$. This is impossible since $K$ is reduced.
Hence   $s_1=1$, so
by (\ref{s1p1}),   $p_1=\rho-2$ and as  $K$ is reduced, we clearly have $s_2\leqslant2^{\rho-2}-1$.
Also, by (\ref{stb}), we have  $2^{\rho-2}-\epsilon\leqslant s_2$.   Therefore,  $s_2=2^{\rho-2}-1$ and $\epsilon=1$. Since $n\geqslant c(r)$ and $k\leqslant b(2\rho-2)$, we have  $p_2=2^{\rho-2}$, $n=c(r)$ and   $k=b(2\rho-2)$.  Thus, by Theorem \ref{bi},  $K$ is isomorphic to ${\cal B}_{\rho-1}$. As $t=\epsilon=1$, it is obvious  that
$\mathrm{rank}(G)=\mathrm{rank}(K)+2$. Therefore $r$ is even and
the definition of $\mathcal{C}_r$ shows  that $G$ is isomorphic to $\mathcal{C}_r$.
}\end{proof}

\section*{Acknowledgments}
This research  was in part supported by grants from IPM to  the first  author (No.\,91050114) and   the second   author (No.\,91050405).
The authors thank anonymous referees for their   valuable comments and suggestions  which dramatically improved the presentation of the article.

{}

\end{document}